\documentclass[twoside,leqno,10pt, A4]{amsart}
\usepackage{amsfonts}
\usepackage{amsmath}
\usepackage{amscd}
\usepackage{amssymb}
\usepackage{amsthm}
\usepackage{amsrefs}
\usepackage{latexsym}
\usepackage{mathrsfs}
\usepackage{bbm}
\usepackage{amscd}
\usepackage{amssymb}
\usepackage{amsthm}
\usepackage{amsrefs}
\usepackage{latexsym}
\usepackage{mathrsfs}
\usepackage{bbm}
\usepackage{enumerate}
\usepackage{graphicx}
\usepackage{color}
\begin{document}

\newtheorem{theorem}[subsection]{Theorem}
\newtheorem{proposition}[subsection]{Proposition}
\newtheorem{lemma}[subsection]{Lemma}
\newtheorem{corollary}[subsection]{Corollary}
\newtheorem{conjecture}[subsection]{Conjecture}
\newtheorem{prop}[subsection]{Proposition}
\newtheorem{defin}[subsection]{Definition}

\numberwithin{equation}{section}
\newcommand{\mr}{\ensuremath{\mathbb R}}
\newcommand{\mc}{\ensuremath{\mathbb C}}
\newcommand{\dif}{\mathrm{d}}
\newcommand{\intz}{\mathbb{Z}}
\newcommand{\ratq}{\mathbb{Q}}
\newcommand{\natn}{\mathbb{N}}
\newcommand{\comc}{\mathbb{C}}
\newcommand{\rear}{\mathbb{R}}
\newcommand{\prip}{\mathbb{P}}
\newcommand{\uph}{\mathbb{H}}
\newcommand{\fief}{\mathbb{F}}
\newcommand{\majorarc}{\mathfrak{M}}
\newcommand{\minorarc}{\mathfrak{m}}
\newcommand{\sings}{\mathfrak{S}}
\newcommand{\fA}{\ensuremath{\mathfrak A}}
\newcommand{\mn}{\ensuremath{\mathbb N}}
\newcommand{\mq}{\ensuremath{\mathbb Q}}
\newcommand{\half}{\tfrac{1}{2}}
\newcommand{\f}{f\times \chi}
\newcommand{\summ}{\mathop{{\sum}^{\star}}}
\newcommand{\chiq}{\chi \bmod q}
\newcommand{\chidb}{\chi \bmod db}
\newcommand{\chid}{\chi \bmod d}
\newcommand{\sym}{\text{sym}^2}
\newcommand{\hhalf}{\tfrac{1}{2}}
\newcommand{\sumstar}{\sideset{}{^*}\sum}
\newcommand{\sumprime}{\sideset{}{'}\sum}
\newcommand{\sumprimeprime}{\sideset{}{''}\sum}
\newcommand{\sumflat}{\sideset{}{^\flat}\sum}
\newcommand{\shortmod}{\ensuremath{\negthickspace \negthickspace \negthickspace \pmod}}
\newcommand{\V}{V\left(\frac{nm}{q^2}\right)}
\newcommand{\sumi}{\mathop{{\sum}^{\dagger}}}
\newcommand{\mz}{\ensuremath{\mathbb Z}}
\newcommand{\leg}[2]{\left(\frac{#1}{#2}\right)}
\newcommand{\muK}{\mu_{\omega}}
\newcommand{\thalf}{\tfrac12}
\newcommand{\lp}{\left(}
\newcommand{\rp}{\right)}
\newcommand{\Lam}{\Lambda_{[i]}}
\newcommand{\lam}{\lambda}
\newcommand{\af}{\mathfrak{a}}
\newcommand{\sw}{S_{[i]}(X,Y;\Phi,\Psi)}
\newcommand{\lz}{\left(}
\newcommand{\pz}{\right)}
\newcommand{\bfrac}[2]{\lz\frac{#1}{#2}\pz}
\newcommand{\odd}{\mathrm{\ primary}}
\newcommand{\even}{\text{ even}}
\newcommand{\res}{\mathrm{Res}}
\newcommand{\sumn}{\sumstar_{(c,1+i)=1}  w\left( \frac {N(c)}X \right)}
\newcommand{\lab}{\left|}
\newcommand{\rab}{\right|}
\newcommand{\Go}{\Gamma_{o}}
\newcommand{\Ge}{\Gamma_{e}}
\newcommand{\M}{\widehat}

\theoremstyle{plain}
\newtheorem{conj}{Conjecture}
\newtheorem{remark}[subsection]{Remark}

\makeatletter
\def\widebreve{\mathpalette\wide@breve}
\def\wide@breve#1#2{\sbox\z@{$#1#2$}%
     \mathop{\vbox{\m@th\ialign{##\crcr
\kern0.08em\brevefill#1{0.8\wd\z@}\crcr\noalign{\nointerlineskip}%
                    $\hss#1#2\hss$\crcr}}}\limits}
\def\brevefill#1#2{$\m@th\sbox\tw@{$#1($}%
  \hss\resizebox{#2}{\wd\tw@}{\rotatebox[origin=c]{90}{\upshape(}}\hss$}
\makeatletter

\title[First moment of central values of quadratic Dirichlet $L$-functions]{First moment of central values of quadratic Dirichlet $L$-functions}

\author[P. Gao]{Peng Gao}
\address{School of Mathematical Sciences, Beihang University, Beijing 100191, China}
\email{penggao@buaa.edu.cn}

\author[L. Zhao]{Liangyi Zhao}
\address{School of Mathematics and Statistics, University of New South Wales, Sydney NSW 2052, Australia}
\email{l.zhao@unsw.edu.au}

\begin{abstract}
We evaluate the first moment of central values of the family of quadratic Dirichlet $L$-functions using the method of double Dirichlet series.  Under the generalized Riemann hypothesis, we prove an asymptotic formula with an error term of size that is the fourth root of that of the primary main term.
\end{abstract}

\maketitle

\noindent {\bf Mathematics Subject Classification (2010)}: 11M06, 11M41  \newline

\noindent {\bf Keywords}:  quadratic Dirichlet $L$-functions, first moment, double Dirichlet series

\section{Introduction}\label{sec 1}

  Moments of central values of families of $L$-functions have been widely studied in the literature as they have many important applications. In this paper, we are interested in the first moment of central values of the family of quadratic Dirichlet $L$-functions. For this family, an asymptotic formula for the first moment was initially obtained by M. Jutila \cite{Jutila} with the main term of size $X \log X$ and an error term of size $O(X^{3/4+\varepsilon})$ for any $\varepsilon>0$. An error term of the same size was later given by A. I. Vinogradov and L. A. Takhtadzhyan in \cite{ViTa}. Using the method of double Dirichlet series, D. Goldfeld and J. Hoffstein \cite{DoHo} improved the error term to $O(X^{19/32 + \varepsilon})$. It is also implicit in \cite{DoHo} that one may obtain an error term of size $O(X^{1/2 + \varepsilon})$ for the smoothed first moment, a result that is achieved via a different approach by M. P. Young \cite{Young1} who utlized a recursive argument. The optimal error term is conjectured to be $O(X^{1/4+\varepsilon})$ in \cite{DoHo} and this been observed in a numerical study conducted by M. W. Alderson and M. O. Rubinstein in \cite{AR12}.  In the function field setting, owing partially to the established truth of the Riemann hypothesis there, the analogous asymptotic formula with an error term of the conjectured size was obtained by A. M. Florea \cite{Florea17}.   \newline

 The method of multiple Dirichlet series is a powerful tool when studying moments of $L$-functions. The success of such method relies
heavily on the analytic properties of these series. In \cite{DoHo}, D. Goldfeld and J. Hoffstein used a double Dirichlet series in their work by treating the variables separately. They applied the theory of Eisenstein series of metaplectic type to obtain analytic continuation of the series in one variable. It was later pointed out by A. Diaconu, D. Goldfeld and J. Hoffstein in \cite{DGH} that there are many advantages in viewing multiple Dirichlet series as functions of several complex variables.  From this point of view, much progress has been made towards understanding analytic properties of various multiple Dirichlet series in \cite{DGH}, including a result on third moment of central values of the family of quadratic Dirichlet $L$-functions. \newline

  For any integer $m \equiv 0, 1 \pmod 4$, let $\chi^{(m)}=\leg {m}{\cdot}$ be the Kronecker symbol defined on \cite[p. 52]{iwakow}.  As usual, $\zeta(s)$ is the Riemann zeta function. For any $L$-function, we write $L^{(c)}$ (resp. $L_{(c)}$) for the function given by the Euler product defining $L$ but omitting those primes dividing (resp. not dividing) $c$. We reserve the letter $p$ for a prime throughout the paper and we  write $L_p$ for $L_{(p)}$ for simplicity. In \cite{Blomer11}, V. Blomer obtained meromorphic continuation to the whole $\mc^2$ for the double Dirichlet series given by
\begin{align*}
 \zeta^{(2)}(2s+2w-1)\sum_{(d,2)=1}\frac{L(s, \chi^{(4d)}\psi)\psi'(d)}{d^w}.
\end{align*}
The primary goal of \cite{Blomer11} is to establish a subconvexity bound of the above double series.  Its analytic properties actually also allow one to evaluate the smoothed first moment of $L(1/2, \chi^{(4d)})$ given by
 \begin{align}
\label{ZetWithCharacters}
  \sum_{(d,2)=1}L(\tfrac 12, \chi^{(4d)})w \bfrac {d}X,
\end{align}
  where $w(t)$ is a non-negative Schwartz function. Under generalized Riemann hypothesis (GRH) and arguing in a manner similar to \eqref{Integral for all characters} below, one is able to obtain an asymptotical formula for the expression above with the error term being the conjectured size $O(X^{1/4+\varepsilon})$.   \newline

  The success of Blomer's method relies on obtaining enough functional equations for the underlying double Dirichlet series. However, it is generally a challenging task to establish meromorphic continuation of other multiple Dirichlet series to the entire complex space this way in order to study the corresponding first moment. Thus, it is desirable to seek for an alternative approach to circumvent this difficulty.  For this, we note that M. \v Cech \cite{Cech1} investigated the $L$-functions ratios conjecture for the case of quadratic Dirichlet $L$-functions using multiple Dirichlet series.  The advantage of this method is that instead of pursuing meromorphic continuation to the entire complex space for the multiple Dirichlet series involved as done in other works, one makes a crucial use of the functional equation of a general (not necessarily primitive) quadratic Dirichlet $L$-function \cite[Proposition 2.3]{Cech1} to extend the multiple Dirichlet series under consideration to a suitable large region for the purpose of the investigation.   \newline

 Motivated by the work of \v Cech, we adapt the approach in \cite{Cech1} to assess the first moment of central values of
a family of quadratic Dirichlet $L$-functions. To state our result, we write $\chi_n$ for the quadratic character $\left(\frac {\cdot}{n} \right)$ for an odd, positive integer $n$.  By the quadratic reciprocity law, $L^{(2)}(s, \chi_n)=L(s,  \chi^{(4n)})$ (resp. $L(s, \chi^{(-4n)})$) if $n \equiv 1 \pmod 4$ (resp. $n \equiv -1 \pmod 4$).  Notice that one can factor every such $m$ uniquely into $m=dl^2$ so that $d$ is a fundamental discriminant, i.e. $d$ is either square-free and $d \equiv 1 \pmod 4$ or $d=4n$ with $n \equiv 2,3 \pmod 4$ and square-free.  It is known (see \cite[Theorem 9.13]{MVa1}) that every primitive quadratic Dirichlet character is of the form $\chi^{(d)}$ for some fundamental discriminant $d$. For such $d$, it follows from \cite[Theorem 4.15]{iwakow} that the function $L(s, \chi^{(d)})$ has an analytic continuation to the entirety of $\comc$.  Thus the same can be said of $L^{(2)}(s,  \chi_n)$. \newline
	
In this paper, we evaluate asymptotically the first moment of the family of quadratic Dirichlet $L$-functions $L^{(2)}(s, \chi_n)$ averaged over all odd, positive $n$. Our main result is as follows.
\begin{theorem}
\label{Theorem for all characters}
Under the notation as above and the truth of GRH, suppose that $w(t)$ is a non-negative Schwartz function and $\widehat w(s)$ is its Mellin transform.  For $1/2>\Re(\alpha)>0$ and any $\varepsilon>0$, we have
\begin{align}
\label{Asymptotic for ratios of all characters}
\begin{split}	
\sum_{\substack{(n,2)=1}}L^{(2)}(\tfrac{1}{2}+\alpha, \chi_{n}) w \bfrac {n}X =&  X\M w(1)\frac {\zeta(1+2\alpha)}{\zeta(2+2\alpha)}\frac {1-2^{-1-2\alpha}}{2(1-2^{-2-2\alpha})}+X^{1-\alpha}\M w(1-\alpha)\frac{\pi^{\alpha}\Gamma (1/2-\alpha)\Gamma (\frac {\alpha}2)}{\Gamma(\frac{1-\alpha}2)\Gamma (\alpha)}\frac{\zeta(1-2\alpha)}{\zeta(2)} \frac{2^{2\alpha}}{6} \\
& \hspace*{3cm} +O\lz(1+|\alpha|)^{5+\varepsilon}X^{1/4+\varepsilon}\pz.
\end{split}
\end{align}
\end{theorem}
	
  Notice that the error term in \eqref{Asymptotic for ratios of all characters} is uniform for $\alpha$, we can therefore take the limit $\alpha \rightarrow 0^+$ to deduce the following asymptotic formula for the smoothed first moment of central values of the family of quadratic Dirichlet $L$-functions under consideration.
\begin{corollary}
\label{Thmfirstmomentatcentral}
		With the notation as above and assuming the truth of GRH, we have, for any $\varepsilon>0$,
\begin{align}
\label{Asymptotic for first moment at central}
\begin{split}			
	& 	\sum_{\substack{(n,2)=1}}L^{(2)}(\tfrac{1}{2}, \chi_{n}) w \bfrac {n}X  = XQ(\log X)+O\lz  X^{1/4+\varepsilon}\pz.
\end{split}
\end{align}
  where $Q$ is a linear polynomial whose coefficients depend only on the absolute constants and $\M w(1)$ and $\M w'(1)$.
\end{corollary}

 Note that our error term above is consistent with the conjecture size given in \cite{DoHo}. The explicit expression of $Q$ is omitted here as our main focus is the error term. The proof of Theorem \ref{Theorem for all characters} requires one to obtain meromorphic continuation of certain double Dirichlet series, which we get by making a crucial use of the functional equation of a general quadratic Dirichlet $L$-function in \cite[Proposition 2.3]{Cech1} to convert the original double Dirichlet series to its dual series which we carefully analyze using the ideas of K. Soundararajan and M. P. Young in \cite{S&Y}. \newline

We remark here that our proof of Theorem \ref{Theorem for all characters} implies that \eqref{Asymptotic for first moment at central} holds with the error term $O\lz  X^{1/2+\varepsilon}\pz$ unconditionally. It may also be applied to study the first moment of the family of quadratic Dirichlet $L$-functions given in \eqref{ZetWithCharacters}.  For this reason, we have included a large sieve result (Lemma~\ref{lem:2.3} below) which shall be applied to control the size of the $L$-values on average without the Lindel\"of hypothesis, although the latter may lead to better error terms.  We shall discuss this further after the proof of Lemma~\ref{Estimate For D(w,t)}. \newline

   Lastly, we point out that the novelty of our method in the paper is that instead of seeking for meromorphic continuation of the underlying multiple Dirichlet series to the entire complex space, we only aim to meromorphically continue such series to a region large enough for our purpose. This provides a great degree of flexibility in our treatment and can be easily adapted to investigate other problems. For example, one may study the first moment of families of primitive quadratic Dirichlet $L$-functions using the functional equation of the primitive quadratic Dirichlet $L$-functions themselves to obtain an asymptotic formula with an error term $O(X^{1/2+\varepsilon})$, recovering a result of M.P. Young \cite{Young1}.  The case for primitive quadratic Hecke $L$-functions over imaginary quadratic number fields of class number one has been implicitly worked out in \cite{G&Zhao2023} and the arguments therein carry over to Dirichlet $L$-functions as well. We also note that the analogous result of A. Florea \cite{Florea17} on the first moment of primitive quadratic $L$-functions over function fields suggests that there may be a secondary main term of size $X^{1/3}$ for the number fields case as well.  Thus, some new insights may be required for an asymptotic formula for the first moment of primitive quadratic Dirichlet $L$-functions, if one desires an error term of size $O(X^{1/4+\varepsilon})$. \newline

We end this section by making the convention that, throughout the paper, $\varepsilon$ denotes a small positive quantity that may not be the same in each appearance and the implied constants in $\ll$ and $O$ can depend on $\varepsilon$.

\section{Preliminaries}
\label{sec 2}

\subsection{Gauss sums}
\label{sec2.4}
    We write $\psi_j=\chi^{(4j)}$ for $j = \pm 1, \pm 2$ where we recall that $\chi^{(d)}=\leg {d}{\cdot} $ is the Kronecker symbol for integers $d \equiv 0, 1 \pmod 4$. Note that each $\psi_j$ is a character modulo $4|j|$.  Let $\psi_0$ stand for the primitive principal character. \newline

  Given any Dirichlet character $\chi$ modulo $n$ and any integer $q$, the Gauss sum $\tau(\chi,q)$ is defined to be
\begin{equation*}
		\tau(\chi,q)=\sum_{j\shortmod n}\chi(j)e \left( \frac {jq}n \right), \quad \mbox{where} \quad  e(z)= \exp(2 \pi i z).
\end{equation*}

For the evaluation of $\tau(\chi,q)$, we cite the following result from \cite[Lemma 2.2]{Cech1}.
\begin{lemma}
\label{Lemma changing Gauss sums}
\begin{enumerate}
\item If $l\equiv1 \pmod 4$, then
\begin{equation*}
		\tau\lz\chi^{(4l)},q\pz=
\begin{cases}
					0,&\hbox{if $(q,2)=1$,}\\
					-2\tau\lz\chi_l,q\pz,&\hbox{if $q\equiv2 \pmod 4$,}\\
					2\tau\lz \chi_l ,q\pz,&\hbox{if $q\equiv0 \pmod 4$.}
\end{cases}
\end{equation*}
			\item If $l \equiv3 \pmod 4$, then
\begin{equation*}
				\tau\lz\chi^{(4l)},q\pz=\begin{cases}
					0,&\hbox{if $2|q$,}\\
					-2i\tau\lz\chi_l,q\pz,&\hbox{if $q\equiv1 \pmod 4$,}\\
					2i\tau\lz\chi_l,q\pz,&\hbox{if $q\equiv3 \pmod 4$.}
				\end{cases}
\end{equation*}
		\end{enumerate}
\end{lemma}
	
	Recall that for an odd positive integer $n$, we define $\chi_n=\left(\frac {\cdot}{n} \right)$. We then define an associated Gauss sum $G\lz\chi_n,q\pz$ by
\begin{align*}
\begin{split}
			G\lz\chi_n,q\pz&=\lz\frac{1-i}{2}+\leg{-1}{n}\frac{1+i}{2}\pz\tau\lz\chi_n,q\pz=\begin{cases}
				\tau\lz\chi_n,q\pz,&\hbox{if $n\equiv1 \pmod 4$,}\\
				-i\tau\lz\chi_n,q\pz,&\hbox{if $n\equiv3\pmod 4$}.
			\end{cases}
\end{split}
\end{align*}

  The advantage of $G\lz\chi_n,q\pz$ over $\tau\lz\chi_n,q\pz$ is that $G\lz\chi_n,q\pz$ is now a multiplicative function of $n$. In fact, upon denoting $\varphi(m)$ for the Euler totient function of $m$, we have the following result from \cite[Lemma 2.3]{sound1} that evaluates $G\lz\chi_n,q\pz$.
\begin{lemma}
\label{lem:Gauss}
   If $(m,n)=1$ then $G(\chi_{mn},q)=G(\chi_m,q)G(\chi_n,q)$. Suppose that $p^a$ is
   the largest power of $p$ dividing $q$ (put $a=\infty$ if $m=0$).
   Then for $k \geq 0$ we have
\begin{equation*}
		G\lz\chi_{p^k},q\pz=\begin{cases}\varphi(p^k),&\hbox{if $k\leq a$, $k$ even,}\\
			0,&\hbox{if $k\leq a$, $k$ odd,}\\
			-p^a,&\hbox{if $k=a+1$, $k$ even,}\\
			\leg{qp^{-a}}{p}p^{a}\sqrt p,&\hbox{if $k=a+1$, $k$ odd,}\\
			0,&\hbox{if $k\geq a+2$}.
		\end{cases}
\end{equation*}
\end{lemma}

\subsection{Functional equations for Dirichlet $L$-functions}
	
	We quote the following functional equation from \cite[Proposition 2.3]{Cech1} concerning all Dirichlet characters $\chi$ modulo $n$, which plays a key role in our proof of Theorem \ref{Theorem for all characters}.
\begin{lemma}
\label{Functional equation with Gauss sums}
		Let $\chi$ be any Dirichlet character modulo $n \neq \square$ such that $\chi(-1)=1$. Then we have
\begin{equation}
\label{Equation functional equation with Gauss sums}
			L(s,\chi)=\frac{\pi^{s-1/2}}{n^s}\frac{\Gamma\bfrac{1-s}{2}}{\Gamma\bfrac {s}2} K(1-s,\chi), \quad \mbox{where} \quad K(s,\chi)=\sum_{q=1}^\infty\frac{\tau(\chi,q)}{q^s}.
\end{equation}
\end{lemma}

\subsection{Bounding $L$-functions}

 For a fixed quadratic character $\psi$ modulo $n$, let $\widehat{\psi}$ be the primitive character that induces $\psi$ so that we have $\widehat\psi=\chi^{(d)}$ for some fundamental discriminant $d|n$ (see \cite[Theorem 9.13]{MVa1}). We gather in this section certain estimations on $L(s, \psi)$ that are necessary in the proof of Theorem \ref{Theorem for all characters}. Most of the estimations here are unconditional, except for the following one, which asserts that when $\Re(s) \geq 1/2+\varepsilon$ for any $\varepsilon>0$,  we have by \cite[Theorem 5.19]{iwakow} that under GRH,
\begin{align}
\label{PgLest1}
\begin{split}
& \big| L( s,  \widehat\psi )\big |^{-1} \ll |sn|^{\varepsilon}.
\end{split}
\end{align}

  Write $n=n_1n_2$ uniquely such that $(n_1, d)=1$ and that $p |n_2 \Rightarrow p|d$. The above notations imply that for any integer $q$,
\begin{align}
\label{Ldecomp}
\begin{split}
 L^{(q)}( s, \psi ) =L( s,  \widehat{\psi}) \prod_{p|qn_1}\left( 1-\frac {\widehat\psi(p)}{p^s} \right).
\end{split}
\end{align}

  Observe that
\begin{align*}
 \Big |1-\frac {\widehat\psi(p)}{p^s}\Big | \leq 2p^{\max (0,-\Re(s))}.
\end{align*}

  We then deduce that
\begin{align}
\label{Lnbound}
\begin{split}
 \prod_{p|qn_1}\left( 1-\frac {\widehat\psi(p)}{p^s} \right) \ll 2^{\omega(q_1n)}(qn_1)^{\max (0,-\Re(s))} \ll (qn_1)^{\max (0,-\Re(s))+\varepsilon},
\end{split}
\end{align}
  where $\omega(n)$ denotes the number of distinct prime factors of $n$ and the last estimation above follows from the well-known bound (see \cite[Theorem 2.10]{MVa1})
\begin{align*}
   \omega(h) \ll \frac {\log h}{\log \log h}, \quad \mbox{for} \quad h \geq 3.
\end{align*}

 When $d$ is a fundamental discriminant, we recall the convexity bound for $L(s, \chi^{(d)})$ (see \cite[Exercise 3, p. 100]{iwakow}) asserts that
\begin{align}
\label{Lconvexbound}
\begin{split}
 L( s, \chi^{(d)} ) \ll
\begin{cases}
 &  \left (|d|(1+|s|) \right)^{(1-\Re(s))/2+\varepsilon}, \quad 0 \leq \Re(s) \leq 1, \\
 &  1, \quad \Re(s)>1.
\end{cases}
\end{split}
\end{align}

  To estimate $L( s,  \chi^{(d)} )$ for $\Re(s) < 0$, we note the following functional equation (see \cite[p. 456]{sound1}) for a primitive even character $\chi$.
\begin{align}
\label{fneqnquad}
  \Lambda(s, \chi^{(d)}) := \Big( \frac {|d|} {\pi} \Big)^{s/2}\Gamma \Big( \frac {s}{2} \Big)L(s, \chi^{(d)})=\Lambda(1-s,  \chi^{(d)}).
\end{align}

  Moreover, Stirling's formula (\cite[(5.113)]{iwakow}) implies that, for constants $a_0$, $b_0$,
\begin{align}
\label{Stirlingratio}
  \frac {\Gamma(a_0(1-s)+ b_0)}{\Gamma (a_0s+ b_0)} \ll (1+|s|)^{a_0(1-2\Re (s))}.
\end{align}

  We also conclude from \eqref{Lconvexbound}--\eqref{Stirlingratio} that
\begin{align}
\label{Lchidbound}
\begin{split}
   L(s, \chi^{(d)}) \ll \begin{cases}
   1 \qquad & \Re(s) >1,\\
   (|d|(1+|s|))^{(1-\Re(s))/2+\varepsilon} \qquad & 0\leq \Re(s) <1,\\
    (|d|(1+|s|))^{1/2-\Re(s)+\varepsilon} \qquad & \Re(s) < 0.
\end{cases}
\end{split}
\end{align}

 From \eqref{Ldecomp}, \eqref{Lnbound} and \eqref{Lchidbound}, we deduce that for all complex numbers $s$,
\begin{align}
\label{Lchibound1}
\begin{split}
 L^{(q)}( s, \psi ) \ll (qn_1)^{\max (0,-\Re(s))+\varepsilon}(n(1+|s|))^{\max \{1/2-\Re (s),(1-\Re(s))/2, 0 \} +\varepsilon}.
\end{split}
\end{align}

    We conclude this section by including the following large sieve result for quadratic Dirichlet $L$-functions, which is a consequence of \cite[Theorem 2]{DRHB}.
\begin{lemma} \label{lem:2.3}
 With the notation as above, let $S(X)$ denote the set of real, primitive characters $\chi$ with conductor not exceeding $X$. Then we have, for any complex number $s$  with $\Re(s) \geq 1/2$ and any $\varepsilon>0$,
\begin{align}
\label{L1estimation}
\sum_{\substack{\chi \in S(X)}} |L(s, \chi)|
\ll & X^{1+\varepsilon} |s|^{1/4+\varepsilon}.
\end{align}
\end{lemma}
\begin{proof}
From \cite[Theorem 2]{DRHB}, we get
\begin{align*}
\sum_{\substack{\chi \in S(X)}} |L(s, \chi)|^4
\ll & (X|s|)^{1+\varepsilon}.
\end{align*}
The lemma now follows from the above and H\"older's inequality.
\end{proof}

\subsection{Some results on multivariable complex functions}
	
   We gather here some results from multivariable complex analysis. We begin with the notation of a tube domain.
\begin{defin}
		An open set $T\subset\mc^n$ is a tube if there is an open set $U\subset\mr^n$ such that $T=\{z\in\mc^n:\ \Re(z)\in U\}.$
\end{defin}
	
   For a set $U\subset\mr^n$, we define $T(U)=U+i\mr^n\subset \mc^n$.  We quote the following Bochner's Tube Theorem \cite{Boc}.
\begin{theorem}
\label{Bochner}
		Let $U\subset\mr^n$ be a connected open set and $f(z)$ a function holomorphic on $T(U)$. Then $f(z)$ has a holomorphic continuation to the convex hull of $T(U)$.
\end{theorem}

 We denote the convex hull of an open set $T\subset\mc^n$ by $\widehat T$.  Our next result is \cite[Proposition C.5]{Cech1} on the modulus of holomorphic continuations of multivariable complex functions.
\begin{prop}
\label{Extending inequalities}
		Assume that $T\subset \mc^n$ is a tube domain, $g,h:T\rightarrow \mc$ are holomorphic functions, and let $\tilde g,\tilde h$ be their holomorphic continuations to $\widehat T$. If  $|g(z)|\leq |h(z)|$ for all $z\in T$ and $h(z)$ is nonzero in $T$, then also $|\tilde g(z)|\leq |\tilde h(z)|$ for all $z\in \widehat T$.
\end{prop}

\section{Proof of Theorem \ref{Theorem for all characters}}

  For $\Re(s)$, $\Re(w)$ sufficiently large, define
\begin{align}
\label{Aswzexp}
\begin{split}
A(s,w)=& \sum_{\substack{(n,2)=1 }}\frac{L^{(2)}(w, \chi_{n})}{n^s}
=\sum_{\substack{(nm,2)=1}}\frac{\chi_n(m)}{m^wn^s}= \sum_{\substack{(m,2)=1}}\frac{L( s,\chi^{(4m)} )}{m^w}.
\end{split}
\end{align}

We shall develop some analytic properties of $A(s,w)$, necessary in establishing Theorem \ref{Theorem for all characters}.  Before delving into this analysis, we make the following observation.  In \cite[Lemma 2]{Blomer11}, Blomer developed a relation between $\mathbf{Z}(s,w)$ (which may be regarded of as an analogue of our $A(s,w)$ in \eqref{Aswzexp}) and $\mathbf{Z}(w,s)$.  It is thus a natural question to ask if such a relation is possible for $A(s,w)$ and $A(w,s)$ using quadratic reciprocity.  As remarked in the introduction of this paper, that the advantage of considering all characters is that we can convert $A(s,w)$ to a dual sum and make use of the properties of the dual sum to obtain a good error term.  If we consider primitive $L$-functions, then it would be possible to obtain a functional equation relating $A(s,w)$ and $A(w,s)$, but then the error term would be much larger.
	
\subsection{First region of absolute convergence of $A(s,w)$}

   Using the first equality in the definition for $A(s,w)$ in \eqref{Aswzexp}, we get
\begin{align} \label{Abound}
\begin{split}
		A(s,w)=& \sum_{\substack{(n,2)=1 }}\frac{L^{(2)}(w, \chi_{n})}{n^s}
= \sum_{\substack{(h,2)=1}}\frac {1}{h^{2s}}\sumstar_{\substack{(n,2)=1}}\frac{L^{(2)}(w,  \chi_{n})\prod_{p | h}(1-\chi_{n}(p)p^{-w}) }{n^s} \\
=& \zeta^{(2)}(2s)\sumstar_{\substack{(n,2)=1}}\frac{L^{(2)}(w,  \chi_{n})}{n^sL^{(2)}(2s+w,  \chi_{n})},
\end{split}
\end{align}
where $\sum^*$ henceforth denotes the sum over square-free integers.  Note that the last equality follows by using the identity (which is obtained
by writing the corresponding series into an Euler product)
\[ \sum_{\substack{(h,2)=1}}\frac {1}{h^{2s}}\prod_{p|h} \left( 1- \frac{\chi_{n}(p)}{p^w} \right) = \frac {\zeta^{(2)}(2s)}{L^{(2)}(2s+w,  \chi_{n})} .\]

Write $\widetilde \chi_n$ for the primitive Dirichlet character that induces $\chi^{(n)}$ (resp. $\chi^{(-n)}$) for $n \equiv 1 \pmod 4$  (resp. for $n \equiv -1 \pmod 4$). Recall that we have $L^{(2)}(w, \chi_n)=L(w, \chi^{(\pm 4n)})$ for $n \equiv \pm 1 \pmod 4$. For a square-free integer $n$, we see that $\widetilde \chi_n=\chi^{(n)}$ (resp. $\widetilde \chi_n=\chi^{(4n)}$ )  is a primitive character modulo $n$ for $n \equiv 1 \pmod 4$ (resp. $n \equiv -1 \pmod 4$).  In either case, for square-free integers $n$,
\begin{align*}
\begin{split}
	\big|L^{(2)}(w, \chi_n)\big|=&	\big|(1-\widetilde\chi_n(2)2^{-w})L(w, \widetilde\chi_n)\big| .
\end{split}
\end{align*}

  It follows from \eqref{Abound} and the above that,
\begin{align}
\begin{split}
\label{Aboundinitial}
		A(s,w)
\ll& |\zeta^{(2)}(2s)|\sumstar_{\substack{(n,2)=1}}\frac{|(1-\widetilde\chi_n(2)2^{-w}) L(w, \widetilde\chi_n)|}{|n^{s}L^{(2)}(2s+w,  \chi_{n})|}.
\end{split}
\end{align}

Now \eqref{L1estimation} and partial summation implies that, except for a simple pole at $w=1$, both sums of the right-hand side expression in \eqref{Aboundinitial} are convergent for $\Re(s)>1$,  $\Re(w) \geq 1/2$ as well as for $\Re(2s)>1$, $\Re(2s+w)>1$, $\Re(s+w)>3/2$, $\Re(w) < 1/2$.  Hence, $A(s,w)$ converges absolutely in the region
\begin{equation*}
		S_0=\{(s,w): \Re(s)>1,\ \Re(2s+w)>1,\ \Re(s+w)>3/2\}.
\end{equation*}

As the condition $\Re(2s+w)>1$ is contained in the other conditions, the description of $S_0$ simplifies to
\begin{equation*}
		S_0=\{(s,w): \Re(s)>1, \ \Re(s+w)>3/2\}.
\end{equation*}

 Next, upon writing $m=m_0m^2_1$ with $m_0$ odd and square-free, we recast the last expression of \eqref{Aswzexp} as
\begin{align}
\label{Sum A(s,w,z) over n}		
A(s,w)=&\sum_{\substack{(m,2)=1}}\frac{L( s, \chi^{(4m)})}{m^w}=\sum_{\substack{(m_1,2)=1}}\frac{1}{m_1^{2w}}\sumstar_{\substack{(m_0,2)=1}}\frac{L( s, \chi^{(4m_0)})\prod_{p | m_1}(1-\chi^{(4m_0)}(p)p^{-s}) }{m_0^{w}}.
\end{align}

 Note that $\chi^{(4m_0)}$ is a primitive character modulo $4m_0$ for $m_0 \equiv -1 \pmod 4$.  Arguing as above by making use of \eqref{L1estimation} and partial summation again, the sum over $m$ such that $m_0 \equiv -1 \pmod 4$ in \eqref{Sum A(s,w,z) over n} converges absolutely in the region
\begin{align*}
		S_1 =& \{(s,w):\hbox{$\Re(w)>1, \ \Re(s+w)\geq \tfrac 32, \ \Re (2s+w) > \tfrac{3}{2} $}\} .
\end{align*}
Note that in the case $m_0 \equiv -1 \pmod{4}$, $m$ is never a square. \newline

  Similarly,  the sum over $m$ with $m_0 \equiv 1 \pmod 4$ in \eqref{Sum A(s,w,z) over n} also converges absolutely.  In this case, $\chi^{(m_0)}$ is a primitive character modulo $m_0$.  This allows us to deduce that the function $A(s,w)$ converges absolutely in the region $S_1$, except for a simple pole at $s=1$ arising from the summands with $m=\square$. \newline

 Notice that the convex hull of $S_0$ and $S_1$ is
\begin{equation}
\label{Region of convergence of A(s,w,z)}
		S_2=\{(s,w):\ \Re(s+w)>\tfrac32, \ \Re(2s+w)>\tfrac32 \}.
\end{equation}

Hence, Theorem \ref{Bochner} implies that $(s-1)(w-1)A(s,w)$ converges absolutely in the region $S_2$.

\subsection{Residue of $A(s,w)$ at $s=1$}
\label{sec:resA}
	
	We see that $A(s,w)$ has a pole at $s=1$ arising from the terms with $m=\square$ from \eqref{Sum A(s,w,z) over n}. In order to compute the corresponding residue and for later use, we define the sum
\begin{align*}
\begin{split}
 A_1(s,w) := \sum_{\substack{(m,2)=1 \\ m =  \square}}\frac{L\lz s, \chi^{(4m)}\pz}{m^w}
=  \sum_{\substack{(m,2)=1 \\ m =  \square}}\frac{\zeta(s)\prod_{p | 2m}(1-p^{-s}) }{m^w} .
\end{split}
\end{align*}

 For any $t \in \mc$, let $a_t(n)$ be the multiplicative function such that $a_t(p^k)=1-1/p^t$ for any prime $p$.  This notation renders
\begin{align*}
\begin{split}
 A_1(s,w)
=  \zeta^{(2)}(s)\sum_{\substack{(m,2)=1 \\ m =  \square}}\frac{a_s(m) }{m^w} .
\end{split}
\end{align*}

Recasting the last sum above as an Euler product,
\begin{align}
\label{residuesgen}
\begin{split}
	 A_1(s,w)
=&  \zeta^{(2)}(s)\prod_{p>2}\sum_{\substack{m \geq0\\m \even}}\frac{a_s(p^{m})}{p^{mw}}= \zeta^{(2)}(s)\prod_{p>2}\lz1+\lz 1-\frac 1{p^s} \pz \frac 1{p^{2w}}(1-p^{-2w})^{-1} \pz \\
=& \zeta^{(2)}(s)\zeta^{(2)}(2w)\prod_{p>2}\lz1-\frac 1{p^{s+2w}}\pz =  \ \zeta(s)\zeta(2w)P(s,w),
\end{split}
\end{align}
  where
\begin{align}
\label{Pdef}
\begin{split}
		P(s, w)=& \lz 1-\frac 1{2^s} \pz \lz 1-\frac 1{2^{2w}} \pz \prod_{p>2}\lz1-\frac 1{p^{s+2w}}\pz.
\end{split}
\end{align}

  It follows from \eqref{residuesgen} and \eqref{Pdef} that except for a simple pole at $s=1$, the functions $P(s, w)$ and $A_1(s,w)$ are holomorphic in the region
\begin{align}
\label{S3}
	S_3=\Big\{(s,w):\ &  \Re(s+2w)>1, \Re (w) > \tfrac{1}{2}  \Big\}.
\end{align}
Note that it may appear from (3.6) that $A_1(s,w)$ has a pole when $w=1/2$.  However, as in the region $S_3$ we have $\Re(w)>1/2$, the function $(s-1)(w-1)(s+w-3/2)A(s,w)$ is holomorphic in $S_3$. \newline

   As the residue of $\zeta(s)$ at $s = 1$ equals $1$, we deduce that
\begin{align}
\label{Residue at s=1}
 \res_{s=1}& A(s, \tfrac{1}{2}+\alpha) = \res_{s=1} A_1(s, \tfrac{1}{2}+\alpha)
=\zeta(1+2\alpha)P(1,\tfrac{1}{2}+\alpha).
\end{align}

\subsection{Second region of absolute convergence of $A(s,w)$}

 We infer from \eqref{Sum A(s,w,z) over n} that
\begin{align}
\begin{split}
\label{A1A2}
 A(s,w) =& \sum_{\substack{(m,2)=1 \\ m =  \square}}\frac{L( s, \chi^{(4m)})}{m^w} +\sum_{\substack{(m,2)=1 \\ m \neq \square}}\frac{L( s, \chi^{(4m)})}{m^w} \\
=&  \sum_{\substack{(m,2)=1 \\ m =  \square}}\frac{\zeta(s)\prod_{p | 2m}(1-p^{-s}) }{m^w} +\sum_{\substack{(m,2)=1 \\ m \neq \square}}\frac{L( s, \chi^{(4m)})}{m^w} := \ A_1(s,w)+A_2(s,w).
\end{split}
\end{align}
   Here we recall from our discussions in the previous section that $A_1(s,w)$ is holomorphic in the region $S_3$, except for a simple pole at $s=1$. \newline

  Next, observe that $\chi^{(4m)}$ is a Dirichlet character modulo $4m$ for any $m\geq1$ such that $\chi^{(4m)}(-1)=1$. We thus apply the functional equation \eqref{Equation functional equation with Gauss sums} given in Lemma \ref{Functional equation with Gauss sums} for $L\lz s, \chi^{(4m)}\pz$ in the case $m \neq \square$, arriving at
\begin{align}
\begin{split}
\label{Functional equation in s}
 A_2(s,w) =\frac{\pi^{s-1/2}}{4^s}\frac {\Gamma (\frac{1-s}2)}{\Gamma(\frac {s}2) } C(1-s,s+w),
\end{split}
\end{align}
  where $C(s,w)$ is given by the double Dirichlet series
\begin{align*}
		C(s,w)=& \sum_{\substack{q, m \\ (m,2)=1 \\ m \neq \square}}\frac{\tau(\chi^{(4m)}, q)}{q^sm^w}=\sum_{\substack{q, m \\ \substack{(m,2)=1}}}\frac{\tau(\chi^{(4m)}, q)}{q^sm^w}-\sum_{\substack{q, m \\ (m,2)=1 \\ m = \square}}\frac{\tau(\chi^{(4m)}, q)}{q^sm^w}.
\end{align*}	

 Note that $C(s,w)$ is initially convergent for $\Re(s)$, $\Re(w)$ large enough by \eqref{Region of convergence of A(s,w,z)}, \eqref{S3} and the functional equation \eqref{Functional equation in s}.  To extend this region, we recast $C(s,w)$ as
\begin{align}
\label{Cexp}
\begin{split}
  C(s,w)=& \sum^{\infty}_{q =1}\frac{1}{q^s}\sum_{\substack{(m,2)=1}}\frac{\tau\lz \chi^{(4m)}, q \pz }{m^w}-\sum^{\infty}_{q =1}\frac{1}{q^s}\sum_{\substack{(m,2)=1 \\ m = \square}}\frac{\tau\lz \chi^{(4m)}, q \pz }{m^w} := \ C_1(s,w)-C_2(s,w).
\end{split}
\end{align}

   For two Dirichlet characters $\psi,\psi'$ whose conductors divide $8$, we define
\begin{align}
\begin{split}
\label{C12def}
	C_1(s,w;\psi,\psi'):= \sum_{l,q\geq 1}\frac{G\lz \chi_l,q\pz\psi(l)\psi'(q) }{l^wq^s} \quad \mbox{and} \quad C_2(s,w;\psi,\psi'):=  \sum_{l,q\geq 1}\frac{G \lz \chi_{l^2},q\pz\psi(l)\psi'(q) }{l^{2w}q^s}.
\end{split}
\end{align}

  We follow the arguments contained in \cite[\S 6.4]{Cech1} and apply Lemma \ref{Lemma changing Gauss sums} to obtain that
\begin{align}
\begin{split}
\label{C(s,w,z) as twisted C(s,w,z)}
		C_1(s,w)=&
			-2^{-s}\big( C_1(s,w;\psi_2,\psi_1)+C_1(s,w;\psi_{-2},\psi_1)\big) +4^{-s}\big( C_1(s,w;\psi_1,\psi_0)+C_1(s,w;\psi_{-1},\psi_0)\big)\\
			&\hspace*{2cm} +C_1(s,w;\psi_1,\psi_{-1})-C_1(s,w;\psi_{-1},\psi_{-1}), \\
   C_2(s,w)=&-2^{1-s}C_2(s,w;\psi_1,\psi_1)+2^{1-2s}C_2(s,w;\psi_1,\psi_0).
\end{split}
\end{align}

  We now follow the approach by K. Soundararajan and M. P. Young in \cite[\S 3.3]{S&Y} to write every integer $q \geq 1$ uniquely as $q=q_1q^2_2$ with $q_1$ square-free to derive that
\begin{equation}
\label{Cidef}
		C_i(s,w;\psi,\psi')=\sumstar_{q_1}\frac{\psi'(q_1)}{q_1^s}\cdot D_i(s, w; q_1,\psi, \psi'), \quad i =1,2,
\end{equation}
	where
\begin{align}
\label{Didef}
\begin{split}
		D_1(s, w; q_1, \psi,\psi')=\sum_{l,q_2=1}^\infty\frac{G\lz \chi_{l},q_1q^2_2\pz\psi(l)\psi'(q^2_2) }{l^wq^{2s}_{2}} \quad \mbox{and} \quad D_2(s, w; q_1, \psi,\psi')=\sum_{l,q_2=1}^\infty\frac{G\lz \chi_{l^2},q_1q^2_2\pz\psi(l)\psi'(q^2_2) }{l^{2w}q^{2s}_{2}}.
\end{split}
\end{align}

  The following result develops the required analytic properties of $D_i(s, w;q_1, \psi, \psi')$.
\begin{lemma}
\label{Estimate For D(w,t)}
 With the notation as above and assuming the truth of GRH, for $\psi \neq \psi_0$, the functions $D_i(s, w; q_1, \psi,\psi')$, with $i=1$, $2$ have meromorphic continuations to the region
\begin{align}
\label{Dregion}
		\{(s,w): \Re(s)>1, \ \Re(w)> \tfrac{3}{4} \}.
\end{align}
    Moreover, the only poles in this region occur for $D_1(s, w; q_1, \psi,\psi')$ at $w = 3/2$ when $q_1=1,  \psi=\psi_1$ or when $q_1=2,  \psi=\psi_2$ and the corresponding pole is simple in either case.  For $\Re(s) \geq 1+\varepsilon$, $\Re(w) \geq 3/4+\varepsilon$, away from the possible
poles, we have
\begin{align}
\label{Diest}
			|D_i(s, w; q_1, \psi,\psi')|\ll (q_1(1+|w|))^{\max \{ (3/2-\Re(w))/2, 0 \}+\varepsilon}.
\end{align}		
\end{lemma}
\begin{proof}
   We focus on $D_1(s, w; q_1, \psi,\psi')$ here since the proof is similar for $D_2(s, w; q_1, \psi,\psi')$ which gives no pole due to the exponent of $2w$ of $l$ in its defintion.  By Lemma \ref{lem:Gauss}, in the double sum in \eqref{Didef} defining $D_1(s,w; q_1, \psi,\psi')$, the summands there are jointly multiplicative functions of $l,q_2$(here we say $S(l,q)$ is jointly multiplicative in $l$ and $q$ if $S(l_1l_2, q_1q_2) = S(l_1,q_1)S(l_2,q_2)$ for $(l_1,l_2)=(q_1,q_2)=(l_1,q_2)=(l_2,q_1)=1$).  Moreover, we may assume that $l$ is odd in \eqref{Didef} since $\psi \neq \psi_0$. These observations enable us to recast $D_1(s,w; q_1, \psi,\psi')$ as an Euler product such that
\begin{align}
\label{D1Eulerprod}
\begin{split}	
 &	D_1(s, w; q_1, \psi,\psi')= \prod_p D_{1,p}(s, w; q_1,  \psi,\psi'),
\end{split}
\end{align}
  where
\begin{align}
\label{Dexp}
\begin{split}	
 &	D_{1,p}(s, w; q_1,  \psi,\psi')= \displaystyle
\begin{cases}
\displaystyle \sum_{k=0}^\infty\frac{ \psi'(2^{2k})}{2^{2ks}}, \quad p=2, \\
\displaystyle \sum_{l,k=0}^\infty\frac{ \psi(p^l)\psi'(p^{2k})G\lz \chi_{p^l}, q_1p^{2k} \pz  }{p^{lw+2ks}}, \quad p>2.
\end{cases}
\end{split}
\end{align}
	
Now, for any fixed $p > 2$,
\begin{align}
\label{Dgenest}	
\begin{split}
  &	\sum_{l,k=0}^\infty\frac{ \psi(p^l)\psi'(p^{2k})G\lz \chi_{p^l}, q_1p^{2k} \pz }{p^{lw+2ks}}  = \sum_{l=0}^\infty \frac{ \psi(p^l)G\lz \chi_{p^l}, q_1 \pz }{p^{lw}}  + \sum_{l \geq 0, k \geq 1}\frac{ \psi(p^l)\psi'(p^{2k})G\lz \chi_{p^l}, q_1p^{2k} \pz }{p^{lw+2ks}}.
\end{split}
\end{align}

Remembering that $q_1$ is square-free, we deduce from Lemma \ref{lem:Gauss} that
\begin{align*}
\begin{split}
	|G( \chi_{p^l}, q_1p^{2k} )| \ll p^l, \quad G(\chi_{p^l}, q_1p^{2k})=0, \quad l \geq 2k+3.
\end{split}
\end{align*}

  The above estimations allow us to see that when $\Re(s)>1$, $\Re(w)>3/4$,
\begin{align}
\label{Dk1est}	
\begin{split}
 \sum_{l \geq 0, k \geq 1}\frac{ \psi(p^l)\psi'(p^{2k})G\lz \chi_{p^l}, q_1p^{2k} \pz }{p^{lw+2ks}} =& \sum_{k \geq 1}\frac{ \psi'(p^{2k})G\lz \chi_{1}, q_1p^{2k} \pz }{p^{2ks}}+\sum_{l, k \geq 1}\frac{ \psi(p^l)\psi'(p^{2k})G\lz \chi_{p^l}, q_1p^{2k} \pz }{p^{lw+2ks}}  \\
\ll & p^{-2\Re(s)}+ \Bigg| \sum^{\infty}_{k=1}\sum_{1 \leq l \leq 2k+2}\frac{1}{p^{l(w-1)+2ks}}\Bigg|  \\
\ll & p^{-2\Re(s)}+\Bigg| \sum^{\infty}_{k=1}\frac{2k+2}{p^{2ks}}\Big (\frac 1{p^{w-1}}+ \frac 1{p^{(2k+2)(w-1)}}\Big ) \Bigg|  \\
\ll & p^{-2\Re(s)}+p^{-2\Re(s)-\Re(w)+1}+p^{-2\Re(s)-4\Re(w)+4} .
\end{split}
\end{align}
Further applying Lemma \ref{lem:Gauss} yields that if $p \nmid 2q_1$ and $\Re(w)>3/4$,
\begin{align}
\label{Dgenl0gen}	
\begin{split}
  & \sum_{l=0}^\infty \frac{ \psi(p^l)G\lz \chi_{p^l}, q_1 \pz }{p^{lw}} = 1+\frac{ \psi(p)\chi^{(q_1)}(p)}{p^{w-1/2}}
= L_p \lz w-\tfrac{1}{2}, \chi^{(q_1)}\psi\pz  \lz 1-\frac {1}{p^{2w-1}} \pz =  \frac {L_{p}\lz w-\tfrac{1}{2}, \chi^{(q_1)}\psi\pz}{\zeta_{p}(2w-1)}.
\end{split}
\end{align}

  We derive from \eqref{Dexp}--\eqref{Dgenl0gen} that for $p \nmid 2q_1$, $\Re(s)>1$, $\Re(w)>3/4$,
\begin{align}
\label{Dgenexp}	
\begin{split}
   & D_{1,p}(s, w; q_1,  \psi,\psi') =  \frac {L_{p}\lz w-\tfrac{1}{2}, \chi^{(q_1)}\psi\pz}{\zeta_{p}(2w-1)}\lz 1+O \Big ( p^{-2\Re(s)}+p^{-2\Re(s)-\Re(w)+1}+p^{-2\Re(s)-4\Re(w)+4} \Big ) \pz .
\end{split}
\end{align}
  Now, we deduce the first assertion of the lemma from \eqref{D1Eulerprod}, \eqref{Dexp} and the above.  We also see this way that the only poles in the region given in \eqref{Dregion} are at $w = 3/2$ and this occurs when $q_1=1$, $\psi=\psi_1$ or when $q_1=2, \psi=\psi_2$. In either case, the corresponding pole is simple.   \newline

  We further note that Lemma \ref{lem:Gauss} implies that when $p | q_1, p \neq 2$,
\begin{align}
\label{Dgenl0}	
\begin{split}
  & \sum_{l=0}^\infty \frac{ \psi(p^l)G\lz \chi_{p^l}, q_1 \pz }{p^{lw}} = 1-\frac{ \psi(p^2)}{p^{2w-1}} =  1+O(p^{-2\Re(w)+1}).
\end{split}
\end{align}

  It follows from \eqref{Dexp}, \eqref{Dk1est} and \eqref{Dgenl0} that for $p | q_1, p \neq 2$, $\Re(s)>1, \Re(w)>3/4$,
\begin{align}
\label{Dgenexp1}	
\begin{split}
   & D_{1,p}(s, w;q_1,  \psi,\psi')
=  1+O \Big (p^{-2\Re(w)+1}+p^{-2\Re(s)}+p^{-2\Re(s)-\Re(w)+1}+p^{-2\Re(s)-4\Re(w)+4}\Big )  .
\end{split}
\end{align}

  We conclude from \eqref{D1Eulerprod}, \eqref{Dexp}, \eqref{Dgenexp} and \eqref{Dgenexp1} that for $\Re(s)\geq 1+\varepsilon$ and $\Re(w) \geq 3/4+\varepsilon$,
\begin{align*}
\begin{split}
	D_{1}(s, w; q_1, \psi,\psi') \ll	& q_1^{\varepsilon}\Big |\frac {L^{(2q_1)}\lz w-\tfrac{1}{2}, \chi^{(q_1)}\psi\pz}{\zeta^{(2q_1)}(2w-1)} \Big | \ll  q_1^{\varepsilon}\Big |\frac {L\lz w-\tfrac{1}{2}, \chi^{(q_1)}\psi\pz}{\zeta(2w-1)} \Big | \ll  (q_1(1+|w|))^{\max \{  (3/2-\Re(w))/2, 0 \}+\varepsilon},
\end{split}
\end{align*}
  where the last bound follows from \eqref{PgLest1} (by taking $\widehat \psi=\psi_0$ to be the
primitive principal character) and \eqref{Lchibound1}.  This leads to the estimate in \eqref{Diest} and completes the proof of the lemma.
\end{proof}

We remark here that the use of the Lindel\"of hypothesis, a consequence of GRH, in the proof of Lemma~\ref{Estimate For D(w,t)} will lead to the improvement of the exponent of $(1+|\alpha|)$ in the $O$-term of \eqref{Asymptotic for ratios of all characters} to $1/4+\varepsilon$.  We keep our computations as unconditional as possible, so that one can deduce an unconditional version of our main result from the arguments. \newline

Now applying Lemma \ref{Estimate For D(w,t)} with \eqref{C(s,w,z) as twisted C(s,w,z)} and \eqref{Cidef}, we infer that $(w-3/2)C(s,w)$ is defined in the region
\begin{equation*}
		\{(s,w):\ \Re(s)>1, \ \Re(w)>3/4, \ \Re(s+w/2)>7/4 \}.
\end{equation*}
	The above together with \eqref{S3} and \eqref{A1A2} now implies that $(s-1)(w-1)(s+w-3/2)A(s,w)$ can be extended to the region
\begin{align} \label{S4def}
		S_4=& \{(s,w): \ \Re(s+2w)>1, \ \Re(s+w)>3/4,\ \Re(w-s)>3/2,  \ \Re(s)<0\}.
\end{align}
Note that the condition $\Re (w+s)>3/4$ and $\Re (s) < 0$ imply that $\Re (w) > 3/4$ in $S_4$, so that we have $S_4 \subset S_3$.  Additionally, the condition $\Re(s+2w)>1$ is redundant.  So
\begin{equation*}
		S_4=\{(s,w):\ \Re(s+w)>3/4,\ \Re(w-s)>3/2, \ \Re(s)<0\}.
\end{equation*}
We remark here that it is possible to repeat the argument of \cite[Section 6.4]{Cech1} to study the analytical property of $C(s,w)$ without relying on the joint multiplicativity of $G(\chi_l, q_1q_2)$ in $l$ and $q_2$.  However, the analogue of $S_4$ obtained this way would be more complicated to describe than what we have above. \newline

As noted below \eqref{S3}, the function $(s-1)(w-1)(s+w-3/2)A(s,w)$ is holomorphic in $S_3$ and hence in $S_4$.  Similarly, $(s-1)(w-1)(s+w-3/2)A(s,w)$ is holomorphic in $S_2$. Then Theorem~\ref{Bochner} implies that $(s-1)(w-1)(s+w-3/2)A(s,w)$ is holomorphic in the convex hull of $S_2 \cup S_4$.  We denoted this convex hull by $S_5$ and
\begin{equation*}
		S_5=\{(s,w):\ \Re(s+w)>3/4 \}.
\end{equation*}
Now Theorem \ref{Bochner} renders that $(s-1)(w-1)(s+w-3/2)A(s,w)$ admits analytic continuation to the region $S_5$.
	
\subsection{Residue of $A(s,w)$ at $s=3/2-w$}
\label{sec:resAw}

  We keep the notation from the proof of Lemma \ref{Estimate For D(w,t)}. We deduce from \eqref{Cexp}, \eqref{C(s,w,z) as twisted C(s,w,z)}, \eqref{Cidef} and Lemma \ref{Estimate For D(w,t)}, that $C(s, w)$ has a pole at $w=3/2$ and
\begin{align} \label{Cres}
\res_{w=3/2}C(s,w)=& 4^{-s}\res_{w=3/2}D_1(s,w;1, \psi_1, \psi_0)+\res_{w=3/2}D_1(s,w;1, \psi_1, \psi_{-1}).
\end{align}
  Here we remark that the pole at $w=3/2$ for $D_1(s, w; q_1, \psi,\psi')$ when $q_1=2, \psi=\psi_2$ does not lead to any pole in $C(s,w)$ since this only affects $C_1(s, w; \psi_2,\psi_1)$ and then we have $\psi_1(q_1)=0$ when $q_1=2$. \newline

     We apply Lemma \ref{lem:Gauss} to see that for $p \neq 2$, 	
\begin{align}
\label{Dk1estpsi1}	
\begin{split}
 \sum_{l \geq 0, k \geq 1}\frac{ \psi'(p^k)G\lz \chi_{p^l}, p^{2k} \pz }{p^{3l/2+2ks}} =& \sum_{k \geq 1}\frac{ G\lz \chi_{1}, p^{2k} \pz }{p^{2ks}}+\sum_{l, k \geq 1}\frac{ G\lz \chi_{p^l}, p^{2k} \pz }{p^{3l/2+2ks}}  \\
= & p^{-2s}(1-p^{-2s})^{-1}+\sum_{k \geq 1}\frac 1{p^{2ks}}\Big (\sum^k_{l=1}\frac{ \varphi(p^{2l}) }{p^{3l}}+\frac {p^{2k}\sqrt{p}}{p^{3(2k+1)/2}} \Big ) \\
= & p^{-2s}(1-p^{-2s})^{-1}+\frac 1p\sum_{k \geq 1}\frac 1{p^{2ks}} = \Big(1+\frac 1p \Big)p^{-2s}(1-p^{-2s})^{-1}.
\end{split}
\end{align}

   We derive from \eqref{Dexp}, \eqref{Dgenest} and \eqref{Dgenl0gen} that for $p \neq 2$,
\begin{align}
\label{D1pexp}	
\begin{split}
   D_{1,p}(s, w; 1,  \psi_1,\psi_0)= D_{1,p}(s, w; 1,  \psi_1,\psi_{-1})=& 1+\frac{1}{p^{w-1/2}}+ \sum_{l \geq 0, k \geq 1}\frac{ \psi'(p^k)G\lz \chi_{p^l}, p^{2k} \pz }{p^{lw+2ks}}
= \zeta_p(w-1/2)Q_p(s,w),
\end{split}
\end{align}
  where, using \eqref{Dk1estpsi1},
\begin{align}
\label{Qpexp}	
\begin{split}
   Q_{p}(s, w)\Big |_{w=3/2} =& \Big(1-\frac {1}{p^{2}} \Big)(1-p^{-2s})^{-1}.
\end{split}
\end{align}
	
 It follows from \eqref{D1Eulerprod}, \eqref{Dexp}, \eqref{D1pexp} and \eqref{Qpexp} that we have
\begin{align}
\label{D1exp}	
\begin{split}
   D_{1}(s, w; 1,  \psi_1,\psi_0) =\zeta(w-1/2)Q(s,w), \quad  D_{1}(s, w; 1,  \psi_1,\psi_{-1})=(1-2^{-2s})\zeta(w-1/2)Q(s,w),
\end{split}
\end{align}
	with
\begin{align}
\label{Qexp}	
\begin{split}
   Q(s,w)\Big |_{w=3/2}=\frac {2\zeta(2s)}{3\zeta(2)}.
\end{split}
\end{align}
	
 As the residue of $\zeta(s)$ at $s=1$ equals $1$, we deduce from \eqref{Cres}, \eqref{D1exp} and \eqref{Qexp} that
\begin{align*}
		\res_{w=3/2}C(s,w)=& \frac {2}{3}\frac {\zeta(2s)}{\zeta(2)}.
\end{align*}
	
Now \eqref{A1A2}, the functional equation \eqref{Functional equation in s} and the above lead to
\begin{align*}
\begin{split}
 \res_{s=3/2-w}A(s,w) =\res_{s=3/2-w}A_2(s,w)=\frac{2 \cdot \pi^{1-w}}{3 \cdot 4^{3/2-w}}\frac {\Gamma (\frac{w-1/2}2)}{\Gamma(\frac {3/2-w}2) } \frac {\zeta(2w-1)}{\zeta(2)}.
\end{split}
\end{align*}

Setting $w=1/2+\alpha$ in the above gives
\begin{align}
\label{Aress1alpha}
\begin{split}
			&\res_{s=1-\alpha}A(s, \tfrac{1}{2}+\alpha) =\frac{2^{2\alpha-1}\pi^{1/2-\alpha}\Gamma (\frac {\alpha}2)}{3\Gamma (\frac {1- \alpha}2)}\frac {\zeta(2\alpha)}{\zeta(2)}.
\end{split}
\end{align}
	Note that the functional equation \eqref{fneqnquad} for $d=1$ implies that
\begin{align*}
  \zeta(2\alpha)=\pi^{2\alpha-1/2}\frac {\Gamma(\tfrac{1}{2}-\alpha)}{\Gamma (\alpha)}\zeta(1-2\alpha).
\end{align*}

   The above allows us to recast the expression in \eqref{Aress1alpha} as
\begin{align}
\label{Aress1}
\begin{split}
		&\res_{s=1-\alpha}A(s,\tfrac{1}{2}+\alpha) =\frac{\pi^{\alpha}\Gamma (\tfrac{1}{2}-\alpha)\Gamma (\frac {\alpha}2)}{\Gamma(\frac{1-\alpha}2)\Gamma (\alpha)}\cdot\frac{\zeta(1-2\alpha)}{\zeta(2)}\cdot \frac{2^{2\alpha}}{6}.
\end{split}
\end{align}

\subsection{Bounding $A(s,w)$ in vertical strips}
\label{Section bound in vertical strips}

 We shall estimate $|A(s,w)|$ in vertical strips, which is necessary in the proof of Theorem \ref{Theorem for all characters}. \newline
	
  For previously defined regions $S_j$, we set
\begin{equation*}
		\widetilde S_j=S_{j,\delta}\cap\{(s,w):\Re(s) \geq -5/2,\ 2 \geq \Re(w) \},
\end{equation*}
	where $\delta$ is a fixed number with $0<\delta <1/1000$ and where $S_{j,\delta}= \{ (s,w)+\delta (1,1) : (s,w) \in S_j \}$ for $j \neq 4$ and $S_{4,\delta}= \{ (s,w)+\delta (-1,1) : (s,w) \in S_4 \}$. We further set
\begin{equation*}
		p(s,w)=(s-1)(w-1)(s+w-3/2)\Gamma(\tfrac w2).
\end{equation*}
Observe that $p(s,w)A(s,w)$ is analytic in the regions under our consideration. \newline

  We note that the function $\Gamma(\frac {1-w}{2})(w-1)$ is also analytic when $\Re(w) \leq 2$.  When $1/2 < \Re(w) \leq 2$, we apply \eqref{Stirlingratio} to see that in this case we have
\begin{align*}
  (w-1)\Gamma(\tfrac w2) \ll \Big|\Gamma(\tfrac {1-w}{2})(w-1)(w-10)^2\Big|.
\end{align*}
Now, \eqref{L1estimation} and partial summation can be used to bound $A(s,w)$ via \eqref{Aboundinitial}.  So in $\widetilde S_0\cap\{(s,w): 2 \geq \Re(w) > 1/2\}$,
\begin{align}
\label{AboundS0}
\begin{split}
      p(s,w)A(s,w) \ll \Big |w^{1/4} (1+100\cdot 2^{-w})(s+5)(w-10)^2(s+w-1/2)\Gamma(\tfrac {1-w}{2})(w-1)\Big |.
\end{split}
\end{align}

For $\Re(w) \leq 1/2$, we apply \eqref{fneqnquad} to estimate $A(s,w)$ from above by way of \eqref{Aboundinitial} to revert to the case $\Re(w) \geq 1/2$.  This gives that the bound in \eqref{AboundS0} continues to hold in $\widetilde S_0\cap\{(s,w): \Re(w) \leq 1/2 \}$.  Thus \eqref{AboundS0} is valid in the entire region $\widetilde S_0$. \newline

   Similarly, we bound the expression for $A(s,w)$ given in \eqref{Sum A(s,w,z) over n}.  In $\widetilde S_1$,
\begin{align*}
	p(s,w)A(s,w) \ll \Big |(1+100\cdot 2^{-w})(s+5)(w-10)^2(s+w-1/2)\Gamma(\tfrac {1-w}{2})(w-1)\Big | \cdot \Big|s+5\Big |^{\max \{1/2-\Re(s), 0\}+\varepsilon}.
\end{align*}
    From this, we apply Proposition \ref{Extending inequalities} by taking $g=p(s,w)A(s,w)$, $h=(1+100\cdot 2^{-w})(s+5)^5(w-10)^3(s+w-1/2)\Gamma(\frac {1-w}{2})(w-1)$ here and note that $h \neq 0$ in $\widetilde S_0 \cup \widetilde S_1$.  We thus deduce that in the convex hull $\widetilde S_2$ of $\widetilde S_0$ and $\widetilde S_1$,
\begin{equation} \label{AboundS2}
		p(s,w)A(s,w) \ll \Big |(1+100\cdot 2^{-w})(s+5)^5(w-10)^2(s+w-1/2)\Gamma(\tfrac {1-w}{2})(w-1)\Big |.
\end{equation}

   Moreover, $\widetilde S_4 \subset \widetilde S_3$ and the conditions $\Re(s+w)>3/4$, $\Re(s)<0$ given in \eqref{S4def} for the definition of $S_4$ imply that $\Re(w)>3/4$ so that $\zeta(2w) \ll 1$ in $\widetilde S_4$.  Then \eqref{Lchibound1} can be used to bound $\zeta(s)$ (corresponding to the case with $\psi=\psi_0$ being the primitive principal character).  This lead to an estimate for $A_1(s,w)$ in \eqref{residuesgen}.  Arguing as above reveals that in the region $\widetilde S_4$, we have
\begin{align}
\label{A1bound}
		p(s,w)A_1(s,w) \ll \Big | (1+100\cdot 2^{-w})(s+5)^5(w-10)^3(s+w-3/2)\Gamma(\tfrac {1-w}{2})(w-1)\Big |.
\end{align}

Also, we deduce from \eqref{Cexp}--\eqref{Didef} and Lemma \ref{Estimate For D(w,t)} that, under GRH,
\begin{equation}
\label{Csbound}
		|C(s,w)|\ll (1+|w|)^{\max \{ (3/2-\Re(w))/2, 0 \}+\varepsilon}
\end{equation}
   in the region
\begin{equation*}
		\{(s,w):\Re(s) \geq 1+\varepsilon, \ \Re(w) \geq 3/4+\varepsilon\}.
\end{equation*}
Now applying \eqref{A1A2}, the functional equation \eqref{Functional equation in s} together with \eqref{Stirlingratio} to bound the ratio of the gamma functions appearing there, as well as the bounds in \eqref{A1bound} and \eqref{Csbound},  we obtain that in the region $\widetilde S_4$, under GRH,
\begin{align}
\label{AboundS3}
\begin{split}
		p(s,w)A(s,w) \ll& \Big |(s+5)(w-10)^2(s+w-1/2)\Gamma(\tfrac {1-w}{2})(w-1)\Big | (1+|s+w|)^{\max \{(3/2-\Re(s+w))/2, 0\} +\varepsilon}(1+|s|)^{3+\varepsilon} \\
\ll & \Big |(s+5)^6(w-10)^4(s+w-1/2)\Gamma(\tfrac {1-w}{2})(w-1)\Big |.
\end{split}
\end{align}

 Lastly, we conclude from \eqref{AboundS2},  \eqref{AboundS3} and Proposition \ref{Extending inequalities} that in the convex hull $\widetilde S_5$ of $\widetilde S_2$ and $\widetilde S_4$, under GRH,
\begin{equation}
\label{AboundS4}
		p(s,w)A(s,w) \ll \Big |(1+100\cdot 2^{-w})(s+5)^6(w-10)^4(s+w-1/2)\Gamma(\tfrac {1-w}{2})(w-1)\Big |.
\end{equation}

\subsection{Completing the proof}

Using the Mellin inversion, we see that for the function $A(s, w)$ defined in \eqref{Aswzexp},
\begin{equation}
\label{Integral for all characters}
		\sum_{\substack{(n,2)=1}}L^{(2)}(\tfrac{1}{2}+\alpha, \chi_{n})w \bfrac {n}X=\frac1{2\pi i}\int\limits_{(2)}A\lz s,\tfrac12+\alpha\pz X^s\widehat w(s) \dif s,
\end{equation}
  where $\widehat{w}$ is the Mellin transform of $w$ given by
\begin{align*}
     \widehat{w}(s) =\int\limits^{\infty}_0w(t)t^s\frac {\dif t}{t}.
\end{align*}
Integration by parts renders that for any integer $E \geq 0$,
\begin{align}
\label{whatbound}
 \widehat w(s)  \ll  \frac{1}{(1+|s|)^{E}}.
\end{align}

We shift the line of integration in \eqref{Integral for all characters} to $\Re(s)=1/4+\varepsilon$.  The integral on the new line can be absorbed into the $O$-term in \eqref{Asymptotic for ratios of all characters} upon using \eqref{Stirlingratio}, \eqref{AboundS4} and \eqref{whatbound}.  We also encounter two simple poles at $s=1$ and $s=1-\alpha$ in the move with the corresponding residues given in \eqref{Residue at s=1} and \eqref{Aress1}, respectively. Direct computations now lead to the main terms given in \eqref{Asymptotic for ratios of all characters}. This completes the proof of Theorem \ref{Theorem for all characters}.

\vspace*{.5cm}

\noindent{\bf Declarations.} We declare that the authors have no competing interests as defined by Springer, or other interests that might be perceived to influence the results and/or discussion reported in this paper.  We further declare that there is no data associated with the results in this article. \newline

\noindent{\bf Acknowledgments.}  The authors are grateful to M. B. Milinovich for some helpful discussions.  P. G. was supported in part by NSFC grant 11871082 and L. Z. by the FRG Grant PS43707 at the University of New South Wales.  Moreover, the authors thank the anonymous referee for his/her very careful reading of the paper and many helpful comments and suggestions.

\bibliography{biblio}
\bibliographystyle{amsxport}

\end{document}